\documentclass[preprint,12pt]{elsarticle}  %use this for latex2e
\usepackage{amssymb}
\usepackage{amsfonts}
\usepackage{amsbsy}
\usepackage{amsmath}

\usepackage{delarray,multirow,array}
\usepackage{color}

\newtheorem{theorem}{Theorem}[section]
\newtheorem{definition}[theorem]{Definition}

\newtheorem{Lemma}[theorem]{Lemma}
\newtheorem{corollary}[theorem]{Corollary}

\newtheorem{remark}[theorem]{Remark}
\newtheorem{example}[theorem]{Example}

\newtheorem{conjecture}{Conjecture}[section]
\def\F{{\mathbb{F}}}

\newenvironment{proof}{\textbf{Proof:}}{\hspace*{\fill}
\nolinebreak\hspace*{\fill}$\Box$\newline\vspace{1mm}}

\newcommand{\sgn}{\rm sgn}

\begin{document}
	
%	\begin{frontmatter}
\title{Superregular matrices over small finite fields}

\author{Paulo Almeida and Diego Napp }
%\institute{}

\address{ Paulo Almeida
	Dept.\ of Mathematics, University of Aveiro, Portugal, palmeida@ua.pt \\
	Diego Napp,  Dept.\ of Mathematics, University of  Alicante, Spain, diego.napp@ua.es \\
}

%\address[rvt]
%\address[focal]{ Departamento de Ingenieria de Sistemas y Automatica, Universidad de Valladolid, 47011 Valladolid, Spain.}

\fntext[Pi]{This work was supported by Portuguese funds through the CIDMA - Center for Research and Development in Mathematics and Applications, and the Portuguese Foundation for Science and Technology (FCT-Funda\c{c}\~ao para a Ci\^encia e a Tecnologia), within project PEst-UID/MAT/04106/2019.
}

%\keywords{1D and 2D convolutional codes, MDS codes, nonsingular matrix}

\begin{abstract}
A trivially zero minor of a matrix is a minor having all its terms in the Leibniz formula equal to zero. A matrix is superregular if all of its minors that are not trivially zero are nonzero. In the area of Coding Theory, superregular matrices over finite fields are connected with codes with optimum distance proprieties. When a superregular matrix has all its entries nonzero, it is called full superregular and these matrices are used to construct Maximum Distance Separable block codes. In the context of convolutional codes, lower triangular Toeplitz superregular matrices are employed to build convolutional codes with optimal column distance. Although full superregular matrices over small fields are known (e.g. Cauchy matrices), the few known general constructions of these matrices having a lower triangular Toeplitz structure require very large field sizes. In this work we investigate lower triangular Toeplitz superregular matrices over small finite prime fields. Following the work of Hutchinson, Smarandache and Trumpf, we study the minimum number of different nontrivial minors that such a matrix have, and exhibit concrete constructions of superregular matrices of this kind.
\end{abstract}

\begin{keyword}
superregular matrix \sep finite prime field \sep  Toeplitz  matrix

2010MSC: 15B33, 15B05, 94B10
\end{keyword}

\maketitle

\section{Introduction and preliminaries}
Let $\F$ denote a finite field, $F= ( \mu_{i,j} )_{1 \leq i,j \leq m} \in \F^{m \times m} $, and let $\mathcal{S}_m$ the symmetric group of order $m$. Recall that the determinant of $F$ is given by
\begin{equation}\label{deter}
|F|=\sum_{\sigma\in\mathcal{S}_m}{\sgn(\sigma)}\mu_{1\sigma(1)}\cdots \mu_{m\sigma(m)},
\end{equation}
where the sign of the permutation $\sigma$, denoted by $\sgn(\sigma)$, is $1$ (resp. $-1$) if $\sigma$ can be written as product of an even (resp. odd) number of transpositions. A {\em trivial term} of the determinant is a term of (\ref{deter}), $\mu_{1\sigma(1)}\cdots \mu_{m\sigma(m)}$, equal to zero. If $F$ is a square submatrix of a matrix $B$, with entries in $\mathbb{F}_{q^M}$, and all the terms of the determinant of $F$ are trivial we say that $|F|$ is a \textit{trivial minor} of $B$. We say that $B$ is \textit{superregular} if all its non-trivial minors are different from zero.

Several notions of superregular matrices have appeared in different areas of mathematics and engineering having in common the specification of some properties regarding their minors \cite{Ando1987,CIM1998,Gan59a,Pinkus2009,Roth1985}. In the context of coding theory these matrices have entries in a finite field $\F$ and are important because they can be used to generate linear codes with good distance properties. A class of these matrices, which we will call {\em full superregular}, were first introduced in the context of block codes. A full superregular matrix is a matrix with all of its minors different from zero and therefore all of its entries nonzero. It is easy to see that a matrix is full superregular if and only if any $\F$-linear combination of $N$ columns (or rows) has at most $N-1$ zero entries. For instance, Cauchy matrices are full superregular and can be used to build the so-called Reed-Solomon block codes. Also, circulant Cauchy matrices can be used to construct MDS codes, see \cite{cl12}. %Note that, obviously, there always exists an $\F$-linear combination of $N$ columns (or rows) with $N-1$ zeros.
It is well-known that a systematic generator matrix $G = [I\ | \ B]$ generates a maximum distance separable (MDS) block code if and only if $B$ is full superregular, \cite{Roth1989}. The $q$-analog of Cauchy superregular matrices has been recently studied in depth in \cite{NERI2020}. \\[1ex]

Convolutional codes are more involved than block codes and, for this reason, a more general class of superregular matrices had to be introduced.

\begin{definition}{\cite[Definition 3.3]{gl03}}
A lower triangular matrix $B$ is  defined to be superregular if all of its minors, with the property that all the entries in their diagonal come from the lower triangular part of $B$, are nonsingular.
\end{definition}

In this paper, we call such matrices {\em LT-superregular}. Note that due to such a lower triangular configuration the remaining minors are necessarily zero. Roughly speaking, superregularity asks for all minors that are possibly nonzero, to be nonzero. In \cite{gl03} it was shown that Toeplitz LT-superregular matrices can be used to construct convolutional codes of rate $k/n$ and degree $\delta$ that are strongly MDS provided that $(n-k)\mid \delta$ (for the rank analog of Toeplitz LT-superregular matrices in the context of the rank metric, see \cite{ALMEIDA2020}). This is again due to the fact that the combination of columns of superregular matrices ensures the largest number of possible nonzero entries for any $\F$-linear combination (for this particular lower triangular structure). In other words, it can be deduced from \cite{gl03} that a lower triangular matrix $B=[b_0 ~ b_1 \dots b_{n-1}]\in \F^{n \times n}$, $b_i$ the columns of $B$, is LT-superregular  if and only if for any $\F$-linear combination $b$ of columns $ b_{i_1}, b_{i_2}, \dots , b_{i_N} $ of $B$, with $i_j < i_{j+1}$, then $wt(b) \geq wt(b_{i_1})-   N +1 = (n-i_1)- N +1$, where $wt(v)$ is the Hamming weight of a vector $v$, i.e., its number or nonzero coordinates. For a similar result but for more general classes of superregular matrices, not necessarily lower triangular, see \cite[Theorem 3.1]{al16}.  \\[1ex]

It is important to note that in this case due to this triangular configuration it is hard to come up with an algebraic construction of LT-superregular matrices.  There exist however two general constructions of these matrices \cite{ANP2013,gl03} although they need very large field sizes. %\section{Size of a finite field have a LT-superregular matrix of order $\gamma$}
In this paper we will be interested in finding Toeplitz LT-superregular matrices over small finite prime fields. So, our matrices will be of the form

\begin{equation}\label{ALT}
A_\gamma=\left [\begin{array}{cccc}
a_1 & 0 & \cdots & 0\\
a_2 & a_1 & \ddots & 0\\
\vdots & \ddots & \ddots & \vdots\\
a_\gamma & \cdots & a_2 & a_1
\end{array}\right ].
\end{equation}

One important question is how large a finite field must be in order that a superregular matrix of a given size can exist over that field. For example, there exists no LT-superregular matrix of order $3$ over the field $\F_2$ because all the entries in the lower triangular part of a superregular matrix must be nonzero, which means that in this case all such entries would have to be $1$; clearly, this does not result in
a superregular matrix, since the lower left submatrix of order $2$ is singular. The size of the smallest finite field for which exists an LT-superregular matrix of order $\gamma\leq 9$ can be seen in Table \ref{minfield}. For $\gamma\geq 10$ the smallest finite field for which exists a Toeplitz LT-superregular matrix of order $\gamma$ is still unknown, but in \cite{Hutchinson2008}, Hutchinson et al. obtained an upper bound for its size and in \cite{HaOs2018} the authors showed the existence of LT-superregular matrices of size $10 \times 10$ over the field $\F_{2^8}$. In \cite[Conjecture 3.5]{Hutchinson2008} and in \cite{gl03} it was conjectured, based on several examples, that an LT-superregular matrix of size $\gamma$ exits over $\F_{2^{\gamma-2}}$ for $\gamma\geq 5$. Recently, new upper bounds on the necessary field size for the existence of these matrices and other superregular matrices with different structure, were presented in \cite{Lieb2019}.

Since the work in \cite{Hutchinson2008} is the motivation for this paper, we will give a brief description of their method to derive an upper bound on the minimum size a finite field must have in order that a superregular matrix of a given size can exist over that field.

Consider
\[X_\gamma=\left [\begin{array}{cccc}
	x_1 & 0 & \cdots & 0\\
	x_2 & x_1 & \ddots & 0\\
	\vdots & \ddots & \ddots & \vdots\\
	x_\gamma & \cdots & x_2 & x_1
\end{array}\right ].\]
a lower triangular Toeplitz matrix with indeterminate entries $x_1, x_2, \dots, x_\gamma$. The determinants of the proper square submatrices of such a matrix are given by nonzero polynomials in these indeterminates. Notice that in any of these polynomials at most the first power of $x_\gamma$ can appear; i.e., each of these polynomials either it is linear in $x_\gamma$ or $x_\gamma$ does not appear in any of its terms. We study now those proper square submatrices of $X_\gamma$ whose determinants are linear in $x_\gamma$. Denote by $L_\gamma$ the set of such submatrices and by $L'_\gamma$, the subset of $L_\gamma$ formed by the submatrices of $X_\gamma$ which are symmetric over the antidiagonal. Hutchinson et al. proved that $N_\gamma:=\frac{1}{2}\left (\mid L_\gamma\mid+\mid L'_\gamma\mid\right )$ is an upper bound for the number of different polynomials that can appear as the determinants of elements of $L_\gamma$. By computer search we found that $N_\gamma$ is actually the exact number of such polynomials, for $\gamma\leq 7$ and $\gamma=9$, but for $\gamma=8$ we have $231$ different polynomials and for $\gamma=10$ we have $2489$ different polynomials, whereas $N_8=232$ and $N_{10}=2494$. In \cite{Hutchinson2008} it is also proved that
	
\begin{equation}\label{Ngamma}
N_\gamma=\frac{\frac{1}{\gamma}\left(\begin{array}{c}2\gamma-2\\
		\gamma-1\end{array} \right)+\left(\begin{array}{c}\gamma-1\\
		\left \lfloor \frac{\gamma-1}{2}\right \rfloor
		\end{array} \right)}{2}.
\end{equation}

Therefore, given $\gamma\geq 1$, a field $\F$ and a lower triangular Toeplitz matrix $A_\gamma\in \F^{\gamma\times \gamma}$ (as in (\ref{ALT})), then $A_\gamma$ has at most $N_\gamma$ different minors that depend on the entry $a_\gamma$, all of them being linear on $a_\gamma$.

\begin{remark}\label{construction}
Notice that $(N_i)_{i=2}^\infty$ is an increasing sequence (and $N_1=N_2=1$). If we choose a field $\F$, such that $|\,\F\,|>N_\gamma$, then we may choose $a_1\in \F$ such that $a_1\neq 0$, then select $a_2\in \F$ such that all the minors involving $a_2$ in the matrix
\[A_2=\left [\begin{array}{cc}
	a_1 & 0 \\
	a_2 & a_1
\end{array}\right ]\]
are nonzero (i.e. any $a_2\neq 0$), then again we can choose $a_3\in \F$ such that all the minors involving $a_3$ in the matrix
\[A_3=\left [\begin{array}{ccc}
a_1 & 0 & 0\\
a_2 & a_1 & 0\\
a_3 & a_2 & a_1
\end{array}\right ]\]
are nonzero, and continuing in this way, we may eventually choose $a_\gamma\in\F$ such that all of the minors involving $a_\gamma$ in the matrix $A_\gamma$ (as in (\ref{ALT})) are nonzero. Therefore, all the non trivial minors of this matrix $A_\gamma$ just constructed, are nonzero. Therefore $A_\gamma$ is LT-superregular. This is the idea of the proof of Theorem \ref{theoSm}.
\end{remark}

\begin{theorem}\cite{Hutchinson2008}\label{theoSm}
Let $\mathbb{F}$ be a finite field such that $\mid\mathbb{F}\mid > N_\gamma$, then there exists a $\gamma\times \gamma$ LT-superregular matrix over $\mathbb{F}$.
\end{theorem}

Unfortunately this upper bound for the minimum field size is not very sharp, as Table \ref{minfield} (obtained in \cite{Hutchinson2008}) demonstrates. The actual minimum field sizes display in the table were obtained by randomised computer search.\\

\begin{table}[h]\begin{center}
	\begin{tabular}{|c|c|c|}\hline
		$\gamma$ & Minimum Field size & Upper Bound ($N_\gamma+1$)\\ \hline
		3 & 3 & 3\\
		4 & 5 & 5\\
		5 & 7 & 11\\
		6 & 11 & 27\\
		7 & 17 & 77\\
		8 & 31 & 233\\
		9 & 59 & 751\\
		10 & $\leq$ 127 & 2495\\ \hline
	\end{tabular}	
\caption{Comparison of actual required field sizes and $N_\gamma+1$}\label{minfield}
\end{center}
\end{table}

In this paper we continue the work initiated in [8] and study lower triangular Toeplitz  superregular matrices over $\F=\F_p$, $p$ an odd prime number. In particular, we investigate the number of different nonzero minors of these matrices. We show that this number is, in many cases, significantly smaller than the $N_\gamma$ derived in \cite{Hutchinson2008} and therefore this immediately improves the upper bound given in Theorem \ref{theoSm} for the minimum field size necessary for the existence of this class of superregular matrices.

\section{Smallest number of different nonzero minors of an LT-superregular Toeplitz matrix}

Since the multiplication by a constant does not change the superregularity of a matrix, we may assume that $a_1=1$. The following lemma implies that we can also assume that $a_2=1$.

\begin{Lemma}{\cite[Theorem 5.8]{to12}} %{\cite[Theorem 4.5]{Hutchinson2006}}
Suppose that the matrix $A_\gamma$ in (\ref{ALT}) is LT-superregular and let $\alpha\in\F^*$. Then the matrix
\[\alpha\otimes A_\gamma=\left [\begin{array}{cccc}
	a_1 & 0 & \cdots & 0\\
	\alpha a_2 & a_1 & \ddots & 0\\
	\vdots & \ddots & \ddots & \vdots\\
	\alpha^{\gamma-1}a_\gamma & \cdots & \alpha a_2 & a_1
\end{array}\right ]\]
is also superregular.
\end{Lemma}

From now on, we will consider $\F=\F_p$ where $p$ is an odd prime number and
\begin{equation}\label{ALTsimples}
A_\gamma=A_\gamma(1, 1, a_3, \dots, a_\gamma)= \left [\begin{array}{ccccc}
1 & 0 & \cdots & \cdots & 0\\
1 & 1 & \ddots &  & \vdots\\
a_3 & 1 & \ddots  & \ddots & \vdots\\
\vdots & \ddots & \ddots & \ddots  & 0\\
a_\gamma & \cdots  & a_3 & 1 & 1
\end{array}\right ].
\end{equation}

In this section, we are interested in studying the smallest possible number of different minors for each $\gamma$ with $3\leq\gamma\leq 9$. For some of the values of $\gamma$ we are able to compute the smallest number of different nonzero minors for every finite prime field. We will also exhibit plenty of superregular matrices for each $3\leq \gamma\leq 9$.

\subsection{$\gamma=3$}
If $\gamma=3$ then there are two minors with the entry $a_3$, namely $|\,[a_3]\,|$ and
$\left |\begin{array}{cc}
1 & 1 \\
a_3 & 1
\end{array}\right |$,
so
\begin{equation}\label{a3not}
a_3\notin S_3=\{0,1\}.
\end{equation}
Hence $\F$ must have at least $3$ elements. Therefore $p\geq 3$. For example, if $a_3\equiv -1\mod{p}$, or $a_3\equiv \frac{1}{2}\mod{p}$, then $A_3(1,1,a_3)$ is LT-superregular.

\subsection{$\gamma=4$}
If $\gamma=4$ then $N_4=4$, i.e., there are four different minors with the entry $a_4$, on the variables $a_3$ and $a_4$, namely
\[|\,[a_4]\,|,\quad \left |\begin{array}{cc}
1 & 1 \\
a_4 & a_3
\end{array}\right |,\quad \left |\begin{array}{cc}
a_3 & 1 \\
a_4 & a_3
\end{array}\right |\quad\mbox{and}\quad \left |\begin{array}{ccc}
1 & 1 & 0\\
a_3 & 1 & 1\\
a_4 & a_3 & 1
\end{array}\right |.\]
So, we must have
\begin{equation}\label{a4not}
a_4\notin S_4=\{0,a_3,a_3^2,2a_3-1\}.
\end{equation}
Notice that if we put $a_3=\frac{1}{2}$ then there are only three different minors involving $a_4$. Hence, in this case we can always choose, for example, $a_4=1$. Therefore, we just proved the following result.

\begin{theorem}\label{a4theo}
For $p\geq 5$ and $a_3\equiv \frac{1}{2}\mod{p}$, the number of different minors of $A_4$ involving $a_4$ is $3$, and if $a_4\notin\{0,\frac{1}{2},\frac{1}{4}\}\mod{p}$ then $A_4(1,1,\frac{1}{2},a_4)$ is LT-superregular.
\end{theorem}

\subsection{$\gamma=5$}
Although $N_5=10$ it is possible to construct LT-superregular matrices $A_5$ over $\F_p$, with $p=7$ as the number of different minors involving $a_5$ can be reduced to $6$ by selecting properly $a_3$ and $a_4$, and to $7$ for all $p\geq 11$. The ten different minors in the variables $a_3, a_4$ and $a_5$ are
\[|\,[a_5]\,|,\quad \left |\begin{array}{cc}
1 & 1 \\
a_5 & a_4
\end{array}\right |,\quad \left |\begin{array}{cc}
a_3 & 1 \\
a_5 & a_4
\end{array}\right |,\quad \left |\begin{array}{cc}
a_3 & 1 \\
a_5 & a_3
\end{array}\right |,\quad \left |\begin{array}{cc}
a_4 & a_3 \\
a_5 & a_4
\end{array}\right |,\]
\[\left |\begin{array}{ccc}
1 & 1 & 0\\
a_3 & 1 & 1\\
a_5 & a_4 & a_3
\end{array}\right |,\quad\left |\begin{array}{ccc}
1 & 1 & 0\\
a_4 & a_3 & 1\\
a_5 & a_4 & a_3
\end{array}\right |,\quad\left |\begin{array}{ccc}
1 & 1 & 0\\
a_4 & a_3 & 1\\
a_5 & a_4 & 1
\end{array}\right |,\quad\left |\begin{array}{ccc}
a_3 & 1 & 1\\
a_4 & a_3 & 1\\
a_5 & a_4 & a_3
\end{array}\right |\quad\mbox{and}\quad\left |\begin{array}{cccc}
1 & 1 & 0 & 0\\
a_3 & 1 & 1 & 0\\
a_4 & a_3 & 1 & 1\\
a_5 & a_4 & a_3 & 1
\end{array}\right |.\]
Now, if $p$ is sufficiently large such that $a_3$ satisfies (\ref{a3not}), $a_4$ satisfies (\ref{a4not}) and
\begin{equation}\label{a5not}\begin{split}
a_5\notin S_5=\left\{0, a_4, a_3a_4, a_3^2, \frac{a_4^2}{a_3},  a_3^2-a_3+a_4,\right . & -a_3^2+a_3a_4+a_4, 2a_4-a_3, \\
 & \left .\frac{a_3^3-2a_3a_4+a_4^2}{a_3-1}, a_3^2-3a_3+2a_4+1\right\},
\end{split}
\end{equation}
then $A_5$ is LT-superregular.

Notice that if $a_3\equiv \frac{1}{4}\mod{p}$ and $a_4\equiv -\frac{1}{8}\mod{p}$ then
\[S_5=\left \{0, -\frac{1}{8}, -\frac{1}{32}, \frac{1}{16}, \frac{1}{16}, -\frac{5}{16}, -\frac{7}{32}, -\frac{1}{2}, -\frac{1}{8}, \frac{1}{16}\right \},\]
i.e.
 \[a_3^2=\frac{a_4^2}{a_3}= a_3^2-3a_3+2a_4+1\equiv \frac{1}{16}\mod{p}\]
and
\[\frac{a_3^3-2a_3a_4+a_4^2}{a_3-1}=a_4\equiv -\frac{1}{8}\mod{p}.\]
In the case $p=7$ we even have $-a_3^2+a_3a_4+a_4\equiv 0\mod{7}$.
Also, if $a_3\equiv \frac{3}{4}\mod{p}$ and $a_4\equiv \frac{3}{8}$ then
\[S_5=\left \{0, \frac{3}{8}, \frac{9}{32}, \frac{9}{16}, \frac{3}{16}, \frac{3}{16}, \frac{3}{32}, 0, 0, \frac{1}{16}\right \},\]
i.e.
\[2a_4-a_3=\frac{a_3^3-2a_3a_4+a_4^2}{a_3-1}=0\]
and
\[\frac{a_4^2}{a_3}=a_3^2-a_3+a_4\equiv\frac{3}{16}\mod{p}.\]
In this case, we also have $a_3a_4\equiv a_3^2-3a_3+2a_4+1\equiv 4\mod{7}$. Moreover, it is easy to see that if $a_5\equiv \frac{1}{4}\mod{p}$ then $a_5\notin S_5$ in both cases.

If we follow the previous subsections and consider $a_3=\frac{1}{2}$ and $a_4=1$, we obtain
\[S_5=\left \{0, 1, \frac{1}{2}, \frac{1}{4}, 2, \frac{3}{4}, \frac{5}{4}, \frac{3}{2}, -\frac{1}{4}, \frac{7}{4}\right \},\]
so there are ten different expressions involving $a_5$ for $p\geq 11$. Nevertheless, we always have $-\frac{1}{2}\notin S_5$.
Therefore we have the following result.

\begin{theorem}
Let $p\geq 7$ and $(a_3,a_4)\in \{\left (\frac{1}{4},-\frac{1}{8}\right ),(\frac{3}{4},\frac{3}{8})\}\mod{p}$.
\begin{enumerate}
	\item If $p=7$ then $A_5(1,1,a_3,a_4,a_5)$ has $6$ different minors involving $a_5$;
	\item If $p>7$ then $A_5(1,1,a_3,a_4,a_5)$ has $7$ different minors involving $a_5$;
	\item If $p\geq 7$ then $A_5(1,1,a_3,a_4,\frac{1}{4})$ is LT-superregular;	
	\item If $p\geq 11$ then $A_5(1,1,\frac{1}{2},1,-\frac{1}{2})$ is LT-superregular.
\end{enumerate}
\end{theorem}
\begin{proof}
The statements $1.$ and $2.$ were obtained above. Since the conditions (\ref{a3not}), (\ref{a4not}) and (\ref{a5not}) are satisfied, the statements $3.$ and $4.$ are also true.
\end{proof}

\subsection{$\gamma=6$}
We have $N_6=26$ and if $p$ is sufficiently large, $a_6\in\F_p$ and
\begin{equation}\label{a6not}
\begin{split}
a_6\notin S_6 = & \left \{0,a_5, a_3a_4, a_4^2, \frac{a_5^2}{a_4}, \frac{a_5^2-2a_3a_4a_5+a_4^3}{a_4-a_3^2}, \frac{a_3a_5-a_4^2+a_4a_5}{a_3}, a_5-a_3a_4+a_4^2, a_5-a_3a_4+a_3a_5, \right .\\
& \frac{a_4a_5+a_3^2a_5-a_3a_4^2-a_5^2}{a_3-a_4}, \frac{a_3a_5+a_4^2-a_3^2a_4-a_4a_5}{1-a_3}, \frac{a_5-2a_3a_4+a_3^3+2a_4^2-a_3^2a_4-a_4a_5}{1-a_3},\\
 & -1+4a_3-3a_4-3a_3^2+2a_5+2a_3a_4, a_3a_5, a_3(2a_5-a_3a_4), a_3(a_4-a_3^2+a_5), \frac{a_4a_5}{a_3},\\
 &  -(a_4-a_3^2-a_5+a_3^3-a_3a_5), -(a_4-a_3^2-2a_5+2a_3a_4-a_4^2), -a_3^2+2a_3a_4, -a_4+2a_5,\\
 & -a_3^2+a_5+a_3a_4, -a_4+a_5+a_3a_4, a_3-2a_4-a_3^2+2a_5+a_3a_4, \\
 & \left . a_3-a_4-2a_3^2+a_5+2a_3a_4,\frac{2a_3a_5+a_4^2-3a_3^2a_4+a_3^4-2a_4a_5-2a_3^2a_5+2a_3a_4^2+a_5^2}{1-2a_3+a_4}\right \},
\end{split}
\end{equation}
then all of the minors involving $a_6$ are nonzero (there are at most $26$ minors). We will show that for any prime $p\geq 11$ we can choose $a_3, a_4$ and $a_5$ such that the number of elements of $S_6$ is at most $14$ (being smaller than $14$ for $93.75\%$ of the primes, because it will only be $14$ if and only if $p\equiv 87,107\mod{120}$, as we will see below).

The following result about quadratic residues will be helpful.
\begin{Lemma}
Let $p>5$ be an odd prime number. Then
\begin{enumerate}
	\item $x^2\equiv -1\mod{p}$ is solvable if and only if $p\equiv 1\mod{4}$;
	\item $x^2\equiv 2\mod{p}$ is solvable if and only if $p\equiv\pm 1\mod{8}$;
	\item $x^2\equiv 3\mod{p}$ is solvable if and only if $p\equiv\pm 1\mod{12}$;
	\item $x^2\equiv -3\mod{p}$ is solvable if and only if $p\equiv 1\mod{3}$;
	\item $x^2\equiv 5\mod{p}$ is solvable if and only if $p\equiv\pm 1\mod{5}$;
\end{enumerate}
\end{Lemma}
\begin{proof}
	$1.$ and $2.$ are well known results. Using the quadratic reciprocity law we easily obtain the remaining statements.
\end{proof}

We are able to state and formally prove a result about the minimum number of elements of $S_6$ for any $p\geq 11$ (computer search using Maple, helped us identify the necessary conditions).

\begin{theorem}\label{theoA6}
Let $p\geq 11$, and for each $u\in\{-3,-1,2,3,5\}$ denote by $\sqrt{u}$ one of the two solutions of $x^2\equiv u\mod{p}$, when they exist. Then
\begin{enumerate}
	\item If $p\equiv 1\mod{4}$ and $(a_3,a_4,a_5)=\left (\frac{1}{2},\frac{1+\sqrt{-1}}{4},\frac{1+2\sqrt{-1}}{8}\right )$ then $\mid S_6\mid =13$, except when $p=13$ or $p=17$, in which case we have $\mid S_6\mid =12$;
	\item If $p\equiv \pm 1\mod{8}$ and $(a_3,a_4,a_5)=\left (\frac{1}{2},\frac{\sqrt{2}+2}{8},\frac{\sqrt{2}+1}{8}\right )$ then $\mid S_6\mid =13$, except when $p=23$, in which case we have $\mid S_6\mid =12$;
	\item If $p\equiv 1\mod{3}$ and $(a_3,a_4,a_5)=\left (\frac{1}{2},\frac{3+\sqrt{-3}}{8},\frac{2+\sqrt{-3}}{8}\right )$ then $\mid S_6\mid =13$, except when $p=13$, in which case we have $\mid S_6\mid =11$;
	\item If $p\equiv\pm 1\mod{5}$ and $(a_3,a_4,a_5)=\left (\frac{1}{2},\frac{1+\sqrt{5}}{8},\frac{\sqrt{5}}{8}\right )$ then $\mid S_6\mid =13$, except when $p=11$, in which case we have $\mid S_6\mid =11$;
	\item If $p\equiv\pm 1\mod{12}$ and $(a_3,a_4,a_5)=\left (\frac{1}{2},\frac{\sqrt{3}-1}{4},\frac{2\sqrt{3}-3}{8}\right )$ then $\mid S_6\mid =14$, except when $p=11$, in which case we have $\mid S_6\mid =10$, when $p=13$, in which case we have $\mid S_6\mid =12$ and $p=23$ or $p=61$ in which case we have $\mid S_6\mid =13$.
\end{enumerate}	
Moreover, if $(a_3,a_4,a_5)$ is any of the vectors above for an appropriate prime $p$, except when $p=11$ and  $(a_3,a_4,a_5)=\left (\frac{1}{2},\frac{1+\sqrt{5}}{8},\frac{\sqrt{5}}{8}\right )$, then $A_6(1,1,a_3,a_4,a_5,a_6)$ is LT-superregular, for any $a_6\notin S_6\mod{p}$.
\end{theorem}
\begin{proof}
For each prime $11\leq p\leq 107$, and using Maple, we found for which values of $(a_3, a_4, a_5)$ we would achieve the minimum of $\mid S_6\mid\mod{p}$ and after identifying which elements of $S_6$ become equal, we deduce the expressions stated in the theorem for $(a_3, a_4, a_5)$. For each prime $11\leq p\leq 107$, there is at least one vector $(a_3,a_4,a_5)$ with $a_3\equiv \frac{1}{2}\mod{p}$ for which $\mid S_6\mid\mod{p}$ is minimal. So we assume $a_3\equiv \frac{1}{2}\mod{p}$. Although, we use Maple to deduce expressions for $a_4$ and $a_5$, the following arguments are valid, for every $p\geq 11$.

Suppose $p\equiv 1\mod{4}$, with the calculations in Maple we found that $\mid S_6\mid\mod{p}$ is minimal when the thirteenth element of $S_6$ is null and the forth and twenty third elements are equal, i.e.
\begin{align*}
-1+4a_3-3a_4-3a_3^2+2a_5+2a_3a_4 & =  0\\
-a_4+a_5+\frac{a_4}{2} & =   a_4^2.
\end{align*}
Solving this system of equations, we obtain
\[(a_4,a_5)\in\left \{\left (\frac{\sqrt{-1}+1}{4},\frac{2\sqrt{-1}+1}{8}\right ),\left (\frac{1-\sqrt{-1}}{4},\frac{1-2\sqrt{-1}}{8}\right )\right \}.\]
Substituting the first solution in $S_6$, we obtain
\[\begin{split}
S_6=\left \{0, \frac{3\sqrt{-1}-1}{16}, \frac{2\sqrt{-1}+1}{8}, \right . & \frac{\sqrt{-1}+1}{8}, \frac{4\sqrt{-1}+1}{16}, \frac{2\sqrt{-1}+1}{16}, \frac{3\sqrt{-1}+1}{16}, \\
& \left .\frac{5\sqrt{-1}+1}{32}, \frac{7\sqrt{-1}+1}{32}, \frac{\sqrt{-1}}{4}, \frac{\sqrt{-1}}{8}, \frac{3\sqrt{-1}}{8}, \frac{3\sqrt{-1}}{16}\right \}.
\end{split}\]
Notice that $S_6$ has at most $13$ elements, for every prime for which $\sqrt{-1}$ exists. If $p=13$ and we take $\sqrt{-1}=5$ then  $\frac{5\sqrt{-1}+1}{32}=0$ and if we take $\sqrt{-1}=8$ then $\frac{7\sqrt{-1}+1}{32}=\frac{3\sqrt{-1}}{8}$, it can be seen that for $p=17$ we also have $\mid S_6\mid=12$, so the statement $1.$ is obtained. Notice that the second solution is also in statement $1.$, since the solutions of $x^2\equiv -1\mod{p}$ are symmetric.

Suppose $p\equiv\pm 1\mod{8}$, with the calculations in Maple we found that $\mid S_6\mid\mod{p}$ is minimal when the thirteenth element of $S_6$ (in the expression (\ref{a6not})) is null and the third and fifth elements are equal, i.e.
\begin{align*}
-1+4a_3-3a_4-3a_3^2+2a_5+2a_3a_4 & =  0\\
 \frac{a_5^2}{a_4} & =  \frac{a_4}{2}.
\end{align*}
Solving this system of equations, we obtain
\[(a_4,a_5)\in\left \{\left (\frac{\sqrt{2}+2}{8},\frac{\sqrt{2}+1}{8}\right ),\left (\frac{2-\sqrt{2}}{8},\frac{1-\sqrt{2}}{8}\right )\right \}.\]
Substituting the first solution in $S_6$, we obtain
\[\begin{split}
S_6=\left \{0, \frac{1}{8}\sqrt{2}, \frac{1}{16}\sqrt{2}, \frac{3}{16}\sqrt{2}, \frac{1}{8}+\frac{3}{32}\sqrt{2}, \right . & \frac{3}{32}+\frac{1}{16}\sqrt{2}, \frac{3}{32}+\frac{3}{32}\sqrt{2}, \frac{1}{8}+\frac{1}{8}\sqrt{2}, \frac{1}{16}+\frac{1}{8}\sqrt{2},\\
&  \left .  \frac{3}{32}+\frac{1}{8}\sqrt{2}, \frac{1}{8}+\frac{1}{16}\sqrt{2}, \frac{1}{16}+\frac{1}{16}\sqrt{2}, \frac{1}{16}+\frac{3}{32}\sqrt{2}\right \}.
\end{split}\]
Here, $S_6$ has also at most $13$ elements, for every prime for which $\sqrt{2}$ exists.
If $p=23$ and we consider $\sqrt{2}=5$ then $\frac{3}{32}+\frac{1}{8}\sqrt{2}=0$ and if we consider $\sqrt{2}=18$, we also obtain $\mid S_6\mid=12$, hence the statement $2.$ is proved. Notice that the second solution is also in statement $2.$, since the solutions of $x^2\equiv 2\mod{p}$ are symmetric.

Suppose $p\equiv 1\mod{3}$ with the calculations in Maple we found that $\mid S_6\mid\mod{p}$ is minimal when the thirteenth element of $S_6$ (in the expression (\ref{a6not})) is null and the fourth and tenth elements are equal, i.e.
\begin{align*}
-1+4a_3-3a_4-3a_3^2+2a_5+2a_3a_4 & =  0\\
\frac{a_4a_5+a_3^2a_5-a_3a_4^2-a_5^2}{a_3-a_4} & =   a_4^2.
\end{align*}
Solving this system of equations, we obtain
\[(a_4,a_5)\in\left \{\left (\frac{\sqrt{-3}+3}{8},\frac{\sqrt{-3}+2}{8}\right ),\left (\frac{3-\sqrt{-3}}{8},\frac{2-\sqrt{-3}}{8}\right ),\left (\frac{1}{4},\frac{1}{8}\right )\right \},\]
but from Theorem \ref{a4theo} we cannot have $a_4=\frac{1}{4}$. Substituting the first solution in $S_6$, we obtain
\[\begin{split} S_6=\left \{0, \right . & \frac{1}{8}+\frac{1}{8}\sqrt{-3}, \frac{3}{16}+\frac{1}{8}\sqrt{-3}, \frac{3}{16}+\frac{3}{16}\sqrt{-3}, \frac{3}{32}+\frac{5}{32}\sqrt{-3}, \frac{1}{4}+\frac{1}{8}\sqrt{-3}, \frac{1}{8}+\frac{1}{16}\sqrt{-3},\\
& \left .\frac{1}{16}+\frac{1}{16}\sqrt{-3}, \frac{3}{16}+\frac{1}{16}\sqrt{-3}, \frac{3}{32}+\frac{3}{32}\sqrt{-3}, \frac{5}{32}+\frac{3}{32}\sqrt{-3}, \frac{5}{32}+\frac{5}{32}\sqrt{-3}, \frac{5}{32}+\frac{11}{96}\sqrt{-3}\right \}.
\end{split}\]
Again, $S_6$ has at most $13$ elements, for every prime for which $\sqrt{-3}$ exists.
As before, it can be seen that if $p=13$ then $\mid S_6\mid=11$ and so, we obtain the statement $3.$. Notice that the second solution is also in statement $3.$, since the solutions of $x^2\equiv -3\mod{p}$ are symmetric.

Suppose $p\equiv\pm 1\mod{5}$, with the calculations in Maple we found that $\mid S_6\mid\mod{p}$ is minimal when the thirteenth element of $S_6$ (in the expression (\ref{a6not})) is null and the third and sixth elements are equal, i.e.
\begin{align*}
-1+4a_3-3a_4-3a_3^2+2a_5+2a_3a_4 & =  0\\
\frac{4(a_4^3-a_4a_5+a_5^2)}{4a_4-1}& =  \frac{a_4}{2}.
\end{align*}
Solving this system of equations, we obtain
\[(a_4,a_5)\in\left \{\left (\frac{\sqrt{5}+1}{8},\frac{\sqrt{5}}{8}\right ),\left (\frac{1-\sqrt{5}}{8},-\frac{\sqrt{5}}{8}\right )\right \}.\]
Substituting the first solution in $S_6$, we obtain
\[\begin{split}
\left \{0, \frac{1}{8}\sqrt{5}, \frac{1}{16}\sqrt{5},\right . & -\frac{5}{32}+\frac{5}{32}\sqrt{5}, -\frac{3}{16}+\frac{3}{16}\sqrt{5}, -\frac{1}{8}+\frac{1}{8}\sqrt{5}, \frac{3}{32}+\frac{1}{32}\sqrt{5},\\ & \left . -\frac{1}{16}+\frac{1}{8}\sqrt{5}, -\frac{1}{16}+\frac{1}{16}\sqrt{5}, \frac{1}{16}+\frac{1}{16}\sqrt{5}, \frac{5}{32}+\frac{1}{32}\sqrt{5}, -\frac{1}{32}+\frac{3}{32}\sqrt{5}, \frac{1}{32}+\frac{3}{32}\sqrt{5}\right \}.
\end{split}\]
So, $S_6$ has at most $13$ elements, for every prime for which $\sqrt{5}$ exists.
Clearly, if $p=11$ then $\mid S_6\mid\leq 11$, but it can be seen that all the elements of $\F_{11}$ are in $S_6\mod{11}$. So, statement $4.$ is obtained. Notice that the second solution is also in statement $4.$, since the solutions of $x^2\equiv 5\mod{p}$ are symmetric.

Suppose $p\equiv\pm 1\mod{12}$, with the calculations in Maple we found that $\mid S_6\mid\mod{p}$ is minimal when the eighth and the thirteenth elements of $S_6$ (in the expression (\ref{a6not})) are null, i.e.
\begin{align*}
a_5-\frac{a_4}{2}+a_4^2 & =0\\
-1+4a_3-3a_4-3a_3^2+2a_5+2a_3a_4 & =0.
\end{align*}
Solving this system of equations, we obtain
\[(a_4,a_5)\in\left \{\left (\frac{\sqrt{3}-1}{4},\frac{2\sqrt{3}-3}{8}\right ),\left (-\frac{\sqrt{3}+1}{4},-\frac{2\sqrt{3}+3}{8}\right )\right \}.\]
Substituting the first solution in $S_6$, we obtain
\[\begin{split}
\left \{0,\right . & -\frac{15}{32}+\frac{9}{32}\sqrt{3}, -\frac{7}{32}+\frac{13}{96}\sqrt{3}, -\frac{3}{4}+\frac{3}{8}\sqrt{3}, -\frac{1}{2}+\frac{1}{4}\sqrt{3}, \frac{1}{4}-\frac{1}{8}\sqrt{3}, \frac{1}{8}-\frac{1}{16}\sqrt{3},\\
& \left . \frac{9}{16}-\frac{5}{16}\sqrt{3}, -\frac{7}{16}+\frac{1}{4}\sqrt{3}, -\frac{3}{8}+\frac{1}{4}\sqrt{3}, -\frac{3}{16}+\frac{1}{8}\sqrt{3}, -\frac{1}{4}+\frac{1}{8}\sqrt{3}, -\frac{1}{8}+\frac{1}{8}\sqrt{3}, -\frac{5}{16}+\frac{3}{16}\sqrt{3}\right \}.
\end{split}\]
This time $S_6$ has at most $14$ elements, for every prime for which $\sqrt{3}$ exists.
If $p=11$ and we consider $\sqrt{3}=5$, then
\begin{align*}
-\frac{7}{32}+\frac{13}{96}\sqrt{3} & =\,\, 0,\\
-\frac{3}{4}+\frac{3}{8}\sqrt{3} & =\, -\frac{3}{16}+\frac{1}{8}\sqrt{3},\\
-\frac{1}{4}+\frac{1}{8}\sqrt{3} & =\, \frac{9}{16}-\frac{5}{16}\sqrt{3},\\
-\frac{1}{8}+\frac{1}{8}\sqrt{3} & =\, \frac{1}{8}-\frac{1}{16}\sqrt{3}.
\end{align*}
The other exceptions can also be obtained and so statement $5.$ is satisfied. Notice that the second solution is also in statement $5.$, since the solutions of $x^2\equiv 3\mod{p}$ are symmetric.

To complete the proof, we need to show that $a_3=\frac{1}{2}\notin S_3$, $a_4\notin S_4$ and $a_5\notin S_5$. Clearly $\frac{1}{2}\notin S_3$. Since $a_3=\frac{1}{2}$ in all cases, then $S_4=\{0,\frac{1}{2},\frac{1}{4}\}$.

If $\frac{1}{4}+\frac{1}{4}\sqrt{-1}\in S_4$ then we would get $\sqrt{-1}\in\{-1,0,1\}$ which is impossible for $p>2$.

If $\frac{1}{4}+\frac{1}{8}\sqrt{2}\in S_4$ then we would get $\sqrt{2}\in\{-2,0,2\}$ which is impossible for $p>2$.

If $\frac{3}{8}+\frac{1}{8}\sqrt{-3}\in S_4$ then we would get $\sqrt{-3}\in\{-3,-1,1\}$ which is impossible for $p>3$.

If $\frac{1}{8}+\frac{1}{8}\sqrt{5}\in S_4$ then we would get $\sqrt{5}\in\{-1,1,3\}$ which is impossible for $p>2$.

If $-\frac{1}{4}+\frac{1}{4}\sqrt{3}\in S_4$ then we would get $\sqrt{3}\in\{1,2,3\}$ which is impossible for $p>3$.

Therefore, $a_4\notin S_4$.

In the case $p\equiv 1\mod{4}$, if $a_4=\frac{1}{4}+\frac{1}{4}\sqrt{-1}$, then
\[S_5=\left \{0, \frac{1}{4}, \frac{1}{2}\sqrt{-1}, \frac{1}{4}\sqrt{-1}, \frac{1}{4}+\frac{1}{2}\sqrt{-1}, \frac{1}{4}+\frac{1}{4}\sqrt{-1}, \frac{1}{8}+\frac{1}{8}\sqrt{-1}, \frac{1}{8}+\frac{3}{8}\sqrt{-1}\right \}.\]
It is not difficult to see that if $a_5=\frac{1}{8}+\frac{1}{4}\sqrt{-1}\in S_5$ then $\sqrt{-1}\in\left \{-\frac{1}{2},0,\frac{1}{2}\right \}$, which is only true if $p=5$.

In the case $p\equiv\pm 1\mod{8}$, if $a_4=\frac{1}{4}+\frac{1}{8}\sqrt{2}$, then
\[S_5=\left \{0, \frac{1}{4}, \frac{1}{4}\sqrt{2}, \frac{1}{8}\sqrt{2}, \frac{1}{4}+\frac{1}{4}\sqrt{2}, \frac{1}{4}+\frac{1}{8}\sqrt{2}, \frac{1}{8}+\frac{1}{16}\sqrt{2}, \frac{1}{8}+\frac{3}{16}\sqrt{2}, \frac{1}{16}+\frac{1}{8}\sqrt{2},\frac{3}{16}+\frac{1}{8}\sqrt{2}\right \}.\]
If $a_5=\frac{1}{8}+\frac{1}{8}\sqrt{2}\in S_5$ then $\sqrt{2}\in\left \{-1,0,1\right \}$, which is never true.

If $p\equiv 1\mod{3}$ and $a_4=\frac{3}{8}+\frac{1}{8}\sqrt{-3}$, then
\[\begin{split}
S_5=\left \{0, \frac{1}{4}, \frac{1}{2}+\frac{1}{4}\sqrt{-3}, \frac{1}{4}+\frac{1}{4}\sqrt{-3}, \frac{1}{8}+\frac{1}{8}\sqrt{-3}, \frac{3}{8}+\frac{1}{8}\sqrt{-3}, \right . & \frac{3}{16}+\frac{1}{16}\sqrt{-3}, \frac{3}{16}+\frac{3}{16}\sqrt{-3},\\
& \left . \frac{5}{16}+\frac{1}{16}\sqrt{-3}, \frac{5}{16}+\frac{3}{16}\sqrt{-3}\right \}.
\end{split}\]
If $a_5=\frac{1}{4}+\frac{1}{8}\sqrt{-3}\in S_5$ then $\sqrt{-3}\in\left \{-2,-1,0,1\right \}$, which is only true if $p=7$.

If $p\equiv\pm 1\mod{5}$ and $a_4=\frac{1}{8}+\frac{1}{8}\sqrt{5}$, then
\[\begin{split}
S_5=\left \{0, \frac{1}{4}, \frac{1}{4}\sqrt{5}, -\frac{3}{16}+\frac{3}{16}\sqrt{5}, -\frac{1}{4}+\frac{1}{4}\sqrt{5}, -\frac{1}{8}+\frac{1}{8}\sqrt{5}, -\frac{1}{16}+\frac{3}{16}\sqrt{5}, \right . & \frac{1}{8}+\frac{1}{8}\sqrt{5},\\
& \left .  \frac{1}{16}+\frac{1}{16}\sqrt{5}, \frac{3}{16}+\frac{1}{16}\sqrt{5}\right \}.
\end{split}\]
If $a_5=\frac{1}{8}\sqrt{5}\in S_5$ then $\sqrt{5}\in\left \{0,1,2,3\right \}$, which is never true.

If $p\equiv\pm 1\mod{12}$ and $a_4=-\frac{1}{4}+\frac{1}{4}\sqrt{3}$, then
\[\begin{split}
S_5=\left \{0, \frac{1}{4}, -1+\frac{1}{2}\sqrt{3}, -\frac{5}{4}+\frac{3}{4}\sqrt{3}, -\frac{5}{8}+\frac{3}{8}\sqrt{3},  \right . & -\frac{3}{4}+\frac{1}{2}\sqrt{3}, -\frac{1}{2}+\frac{1}{4}\sqrt{3}, \\
& \left . -\frac{1}{4}+\frac{1}{4}\sqrt{3}, -\frac{1}{8}+\frac{1}{8}\sqrt{3}, \frac{1}{2}-\frac{1}{4}\sqrt{3}\right \}.
\end{split}\]
If $a_5=-\frac{3}{8}+\frac{1}{4}\sqrt{3}\in S_5$ then $\sqrt{3}\in\left \{\frac{3}{2},\frac{7}{4},2,\frac{5}{2}\right \}$, which is never true.

Hence, we obtain the last statement.
\end{proof}

\begin{table}[t]
	\begin{center}
		\begin{tabular}{|c|c|c|c|}\hline
			$\gamma$ & $N_\gamma+1$ & Number of minors & prime field sizes\\\hline
			$6$ & $27$  & \(\begin{tabular}{c}
			10\\
			11\\
			12 \\
			13 \\
			14
			\end{tabular}\)  & \begin{tabular}{c}
				11\\
				13\\
				17 \mbox{or} 23\\
				$p\geq 19$, $p\neq 23$\mbox{ and } $p\not\equiv 83,107\mod{120}$\\
				$p\equiv 83,107\mod{120}$\\
			\end{tabular}\\\hline
		\end{tabular}
	\end{center}
	\caption{Number of minors of $A_6$ involving $a_6$ for each prime field}\label{tabA6}
\end{table}

\begin{remark}
It is not difficult to obtain from Theorem \ref{theoA6} the minimum number of minors of $A_6$ involving $a_6$, for each appropriate prime $p$. These numbers are detailed in Table \ref{tabA6}.
\end{remark}

Theorem \ref{theoA6} shows that, for each prime $p\geq 11$, whenever we choose $a_6$ appropriately, we obtain an LT-superregular matrix. The next question is if it is possible to choose $a_6$ so that $A_6(1,1,\frac{1}{2},a_4,a_5,a_6)$ is LT-superregular for all, or at least many, of the primes in each of the arithmetic progressions above. The next result answers this question.

\begin{corollary}
	Suppose we have $a_3, a_4$ and $a_5$ as in Theorem \ref{theoA6}, for each of the arithmetic progressions considered. Then
	\begin{enumerate}
		\item If $p=11$ then the matrix $A_6(1,1,6,1,5,4)$ is LT-superregular.
		\item if $p=13$ then all the matrices $A_6(1,1,7,8,3,2)$, $A_6(1,1,7,4,12,9)$, $A_6(1,1,7,6,1,2)$ and $A_6(1,1,7,6,1,4)$ are LT-superregular.
		\item if $p\equiv 1\mod{4}$ and $p\geq 17$, take $a_6=\frac{1}{4}$ (if $p=37$, consider $\sqrt{-1}=6$ in the expressions of $a_4$ and $a_5$). Then $A_6(1,1,\frac{1}{2},a_4,a_5,a_6)$ is LT-superregular.
		\item if $p\equiv\pm 1\mod{8}$ and $p\geq 17$, take $a_6=\frac{1}{4}$ (if $p=17$, consider $\sqrt{2}=6$ in the expressions of $a_4$ and $a_5$). Then $A_6(1,1,\frac{1}{2},a_4,a_5,a_6)$ is LT-superregular.
		\item if $p\equiv 1\mod{3}$ and $p\geq 19$ ,take $a_6=\frac{1}{4}$ (if $p=37$, consider $\sqrt{-3}=16$ in the expressions of $a_4$ and $a_5$). Then $A_6(1,1,\frac{1}{2},a_4,a_5,a_6)$ is LT-superregular.
		\item if $p\equiv\pm 1\mod{5}$ and $p\geq 19$, take $a_6=\frac{1}{4}$. Then $A_6(1,1,\frac{1}{2},a_4,a_5,a_6)$ is LT-superregular.
		\item if $p\equiv\pm 1\mod{12}$, $p\geq 23$ and
		\begin{description}
			\item[a)] $p\neq 37$, take $a_6=\frac{1}{4}$ (if $p=23$ consider $\sqrt{3}=16$ and if $p=73$ consider $\sqrt{3}=52$, in the expressions of $a_4$ and $a_5$). Then $A_6(1,1,\frac{1}{2},a_4,a_5,a_6)$ is LT-superregular.
			\item[b)] $p=37$, take $a_6=10$. Then $A_6(1,1,\frac{1}{2},a_4,a_5,a_6)$ is LT-superregular.
		\end{description}	
	\end{enumerate}
\end{corollary}
\begin{proof}
	In the cases $p=11$ or $p=13$, we just wrote all the possibilities for which $a_6\notin S_6$.
	
	If $p\equiv 1\mod{4}$, with $p\geq 17$, the only instance that $\frac{1}{4}\in S_6$ is when $p=37$ and $\sqrt{-1}=31$, because $\frac{1}{4}=\frac{1}{32}+\frac{5}{32}\sqrt{-1}$.
	
	In the case $p\equiv\pm 1\mod{8}$, with $p\geq 17$, the only instance that $\frac{1}{4}\in S_6$ is when $p=17$ and $\sqrt{2}=11$, because $\frac{1}{4}=\frac{3}{32}+\frac{1}{16}\sqrt{2}$.
	
	In the case $p\equiv 1\mod{3}$, with $p\geq 17$, the only instance that $\frac{1}{4}\in S_6$ is when $p=37$ and $\sqrt{-3}=21$, because $\frac{1}{4}=\frac{5}{32}+\frac{11}{96}\sqrt{-3}$.
	
	If $p\equiv\pm 1\mod{5}$ and $p\geq 19$ then $\frac{1}{4}\notin S_6$.
	
	When $p\equiv\pm 1\mod{12}$, there are a few instances when $\frac{1}{4}\in S_6$. If $p=23$ then we must choose $\sqrt{3}=16$, since when $\sqrt{3}=7$, $\frac{1}{4}=-\frac{7}{32}+\frac{13}{96}\sqrt{3}$. If $p=73$ then we must choose $\sqrt{3}=52$, since when $\sqrt{3}=21$, $\frac{1}{4}=-\frac{7}{16}+\frac{1}{4}\sqrt{3}$. If $p=37$, we always have $\frac{1}{4}\in S_6$, because if we choose $\sqrt{3}=15$, then $\frac{1}{4}=-\frac{3}{4}+\frac{3}{8}\sqrt{3}$ and if we choose $\sqrt{3}=22$, then $\frac{1}{4}=-\frac{3}{16}+\frac{1}{8}\sqrt{3}$. Nevertheless, in this case we may take $a_6=10$ (and $a_3=19$, $a_4=33$ and $a_5=19$).
\end{proof}

\begin{remark}
	There are other possibilities for $a_6$ that make $A_6(1,1,\frac{1}{2},a_4,a_5,a_6)$ LT-superregular for many primes. For example, if $p\geq 17$, with $p\equiv 1\mod{4}$ and we choose $a_6=\frac{\sqrt{-1}}{2}$ with $\sqrt{-1}<\frac{p}{2}$ then $a_6\notin S_6$. Therefore, $A_6(1,1,\frac{1}{2},a_4,a_5,a_6)$ is LT-superregular. Notice that if $p=13$ then $a_6\in S_6$. If we chose $\sqrt{-1}>\frac{p}{2}$, then $a_6\in S_6$ when $p\in\{13,17,37,41,61\}$. More explicitly,
	
	if $p=13$ and $\sqrt{-1}=8$ then $\frac{1}{2}\sqrt{-1}=\frac{1}{16}+\frac{3}{16}\sqrt{-1}$;
	
	if $p=17$ and  $\sqrt{-1}=13$ then $\frac{1}{2}\sqrt{-1}=\frac{1}{16}+\frac{1}{4}\sqrt{-1}$;
	
	if $p=37$ and  $\sqrt{-1}=31$ then $\frac{1}{2}\sqrt{-1}=\frac{1}{16}+\frac{1}{8}\sqrt{-1}$;
	
	if $p=41$ and  $\sqrt{-1}=32$ then $\frac{1}{2}\sqrt{-1}=\frac{1}{32}+\frac{7}{32}\sqrt{-1}$;
	
	and if $p=61$ and  $\sqrt{-1}=50$ then $\frac{1}{2}\sqrt{-1}=\frac{1}{32}+\frac{5}{32}\sqrt{-1}$.
\end{remark}

We can create more examples of LT-superregular matrices using the values for $a_3, a_4$ and $a_5$ from the previous subsections.

\begin{example}
If we take $a_3=\frac{1}{4}$, $a_4=-\frac{1}{8}$ and $a_5=\frac{1}{4}$ then
\[S_6=\left \{0,-\frac{13}{32}, -\frac{7}{32}, -\frac{1}{2}, -\frac{1}{8}, -\frac{1}{32}, \frac{1}{4}, \frac{1}{16}, \frac{1}{64}, \frac{5}{8}, \frac{5}{32}, \frac{7}{16}, \frac{11}{32}, \frac{17}{32}, \frac{17}{128}, \frac{19}{64}, \frac{23}{32}, \frac{29}{32}, \frac{31}{64}, \frac{49}{64} \right \}\]
has at most $20$ elements. So, if $p\geq 23$ and $a_6=-\frac{1}{4}$, which is not in $S_6$, for any $p\geq 23$, then $A_6(1,1,a_3,a_4,a_5,a_6)$ is LT-superregular.

If we take $a_3=\frac{3}{4}$, $a_4=\frac{3}{8}$ and $a_5=\frac{1}{4}$ then
\[S_6=\left \{0,-\frac{1}{32}, \frac{1}{4}, \frac{1}{6}, \frac{1}{8}, \frac{1}{16}, \frac{3}{16}, \frac{3}{32}, \frac{3}{64}, \frac{5}{32}, \frac{7}{32}, \frac{7}{64}, \frac{9}{32}, \frac{9}{64}, \frac{11}{32}, \frac{13}{64}, \frac{13}{96}, \frac{17}{64}, \frac{17}{96}, \frac{21}{128} \right \}\]
has at most $20$ elements. So, if $p\geq 23$ and $a_6=-\frac{5}{16}$, which is not in $S_6$, for any $p\geq 23$, then $A_6(1,1,a_3,a_4,a_5,a_6)$ is LT-superregular.

If we take $a_3=\frac{1}{2}$, $a_4=1$ and $a_5=-\frac{1}{2}$ then
\[S_6=\left \{-2, -1, 0, 1, 2, -\frac{13}{8}, -\frac{11}{4}, -\frac{9}{4}, -\frac{7}{2}, -\frac{7}{4}, -\frac{5}{4}, -\frac{3}{4}, -\frac{1}{2}, -\frac{1}{4}, \frac{1}{2}, \frac{1}{4}, \frac{1}{8}, \frac{3}{4}, \frac{7}{3}, \frac{7}{4}, \frac{11}{4}, \frac{37}{16} \right \}\]
has at most $22$ elements. So, if $p\geq 23$ and $a_6=\frac{3}{2}$, which is not in $S_6$, for any $p\geq 23$, then $A_6(1,1,a_3,a_4,a_5,a_6)$ is LT-superregular.
\end{example}

\subsection{When $\gamma\geq 7$}
\begin{table}[h]
	\begin{center}
		\begin{tabular}{|c|c|c|c|}\hline
			$\gamma$ & $N_\gamma+1$ & Different minors & Field size \\ \hline
			7 & 77 & \(\begin{tabular}{c}
			16\\
			18\\
			21 \\
			24 \\
			25 \\
			28 \\
			29 \\
			30 \\
			31  \\
			32 \\
			35 \\
			36 \\
			37 \\
			38 \\
			39
			\end{tabular}\)  &\(\begin{tabular}{c}
			17\\
			19\\
			23\\
			29\\
			31 \\
			37 \\
			41 \\
			47 \\
			43 \\
			53 \\
			59 \mbox{or} 61 \\
			67 \mbox{or} 73 \\
			71 \\
			79 \\
			83, 89 \mbox{or} 97 \\
			\end{tabular}\) \\ \hline
		\end{tabular}	
	\end{center}
	\caption{Minimum number of different minors involving $a_7$ for $17\leq p\leq 97.$}\label{tabA7}
\end{table}

For $\gamma\geq 7$, the count of the minimum number of different minors involving $a_\gamma$ for every prime field $\F_p$ for which $A_\gamma$ is LT-superregular, gets much more complicated, as there are too many different values. Therefore we chose to construct examples of LT-superregular matrices for some of the finite prime fields for which $A_\gamma$ is LT-superregular, for each $7\leq \gamma\leq 10$. For each prime $p$ and each $\gamma$ we created the sets $S_\mu$, for $\mu\leq\gamma$ and tried recursively, using Maple, all the vectors $(a_3, a_4, \dots, a_{\gamma-1})\in\F_p^{\gamma-3}$, that satisfied $a_i\notin S_i$, for $3\leq i\leq\gamma-1$ in order to find the vectors $(a_3, a_4, \dots, a_{\gamma-1})$ that made $\mid S_\gamma\mid$ smallest.

Suppose $\gamma=7$ then $N_\gamma=76$. But as one can see from the Table \ref{tabA7}, there are too many minimum numbers of different minors of $A_7$ involving $a_7$. The smallest finite prime fields that have an LT-superregular matrix of order $7$ have all different minimum numbers and we were not able to find a pattern from which we could deduce general sequences as we did in the case $\gamma=6$. Nevertheless, we are able to exhibit LT-superregular matrices for every $p\geq 17$.

If $p=17$ there are $8$ LT-superregular matrices $A_7$, one of which is $A_7(1,1,9,3,5,1,3)$. For this example, $\mid S_4\mid=3$,  $\mid S_5\mid=8$, $\mid S_6\mid=13$ and $\mid S_7\mid=16$. If $p=19$ there are $82$ LT-superregular matrices $A_7$, one of which is $A_7(1,1,10,13,1,18,7)$. The number of elements of $S_4, S_5$ and $S_6$ are also $3, 8$ and $13$ respectively, and $\mid S_7\mid=18$. If $p=23$ there are only two examples of LT-superregular matrices, which are $A_7(1,1,4,19,6,4,8)$ and $A_7(1,1,4,19,6,4,15)$. It is interesting to notice that $S_4$ and $S_5$ achieve the maximum number of elements in these two examples while $S_6$ has $17$ elements. Hence, sometimes $S_\gamma$ has the minimum number of elements when some of the $S_\mu$, with $\mu<\gamma$ have the maximum.
 \begin{table}[h]
 	\begin{center}
 		\begin{tabular}{|c|c|c|c|}\hline
 			$\gamma$ & $N_\gamma+1$ & $(a_3,a_4,a_5,a_6)$ & $|\,S_7\,|$ \\ \hline
 			7 & 77 & \(\begin{tabular}{c}
 			\\
 			$\displaystyle{\left (\frac{1}{2},1,-\frac{1}{2},\frac{3}{2}\right )}$ \\
 			\\
 			$\displaystyle{\left (\frac{1}{4},-\frac{1}{8},\frac{1}{4},-\frac{1}{4}\right )}$\\
 			\\
 			$\displaystyle{\left (\frac{3}{4},\frac{3}{8},\frac{1}{4},-\frac{5}{16}\right )}$\\
 			\\
 			$\displaystyle{\left (\frac{1}{2},\frac{1}{4}+\frac{1}{4}\sqrt{-1},\frac{1}{8}+\frac{1}{4}\sqrt{-1},\frac{1}{4}\right )}$ \\
 			\\
 			$\displaystyle{\left (\frac{1}{2},\frac{1}{4}+\frac{1}{8}\sqrt{2},\frac{1}{8}+\frac{1}{8}\sqrt{2},\frac{1}{4}\right )}$ \\
 			\\
 			$\displaystyle{\left (\frac{1}{2},\frac{3}{8}+\frac{1}{8}\sqrt{-3},\frac{1}{4}+\frac{1}{8}\sqrt{-3},\frac{1}{4}\right )}$ \\
 			\\
 			$\displaystyle{\left (\frac{1}{2},\frac{1}{8}+\frac{1}{8}\sqrt{5},\frac{1}{8}\sqrt{5},\frac{1}{4}\right )}$ \\
 			\\
 			$\displaystyle{\left (\frac{1}{2},-\frac{1}{4}+\frac{1}{4}\sqrt{3},-\frac{3}{8}+\frac{1}{4}\sqrt{3},\frac{1}{4}\right )}$ \\
 			$\,$
 			\end{tabular}\)  &\(\begin{tabular}{c}
 			\\
 			56\\
 			\\
 			\\
 			67\\
 			\\
 			\\
 			68\\
 			\\
 			\\
 			65\\
 			\\
 			\\
 			57\\
 			\\
 			\\
 			65\\
 			\\
 			\\
 			55 \\
 			\\
 			\\
 			71\\
 			$\,$
 			\end{tabular}\) \\ \hline
 		\end{tabular}	
 	\end{center}
 	\caption{maximum size of $S_7$ for some $(a_3,a_4,a_5,a_6)$.}\label{tabA7seq}
 \end{table}

 If we use the vectors already considered in the previous sections, we obtain very large values for the number of elements of $S_7$ (see Table \ref{tabA7seq}), in comparison to the ones obtained in Table \ref{tabA7}. So, considering $\mu<\nu\leq\gamma$, having $|\,S_\mu\,|$ small for some $(a_3, \dots, a_{\mu-1})$ doesn't imply that $|\,S_\nu\,|$ is also small for  $(a_3, \dots, a_{\mu-1},a_\mu, \dots, a_\nu)$. Nevertheless, the sequences in Table \ref{tabA7seq} can be used to construct LT-superregular matrices of order $7$, for finite prime fields, when $p\geq 59$, since $59,61$ and $71$ are congruent with plus or minus one module $5$ and $67\equiv 1\mod{3}$. The Table \ref{tabA7small} shows examples of vectors  $(a_3,a_4,a_5,a_6,a_7)$ such that the size of $S_7$ is minimum and from which we can create $7\times 7$ LT-superregular matrices when $p< 59$.
 \begin{table}[h]
 	\begin{center}
 		\begin{tabular}{|c|c|}\hline
Field size & Example of $(a_3,a_4, a_5, a_6 ,a_7)$ \\ \hline
17 &  $(9,3,5,1,3)$\\
19 &  $(10,13,1,18,7)$\\
23 &  $(4,19,6,4,8)$ \\
29 &  $(15, 19, 8, 22,1)$\\
31 &  $(4, 30, 22, 17, 2)$\\
37 & $(8, 35, 6, 25,12)$ \\
41 & $(7, 22, 26, 7,1)$ \\
43 & $(17, 12, 25, 23,2)$ \\
47 & $(24, 9, 3, 18,4)$ \\
53 & $(27, 42, 22, 20,3)$\\ \hline
 \end{tabular}	
 	\end{center}
 	\caption{Examples of LT-superregular matrices of order $7$ for small finite prime fields}\label{tabA7small}
 \end{table}

 \begin{table}[h]
 	\begin{center}
 		\begin{tabular}{|c|c|c|c|}\hline
 			$\gamma$ &		Field size & Example of $(a_3,a_4, \dots,a_\gamma)$ & Different minors\\ \hline
 			8 & 31 & $(7, 22, 20, 2, 13,5)$ & 30\\
 			& 37 & $(2, 8, 28, 32, 18,16)$ & 36\\
 			9 & 59 & $(5,28,58,56,26,18,19)$ & 58\\
 			&	61 & $(7,60,55,39,10,12,16)$ & 60\\\hline
 		\end{tabular}	
 	\end{center}
 	\caption{Examples of LT-superregular matrices of order $8$ and $9$ for small finite prime fields}\label{tabA}
 \end{table}

Again, using Maple we were able to compute LT-superregular matrices of order $\gamma$, for $\gamma=8$ and $\gamma=9$ over the two smallest finite prime fields. These examples are shown in Table \ref{tabA}. In \cite{HaOs2018}, the authors presented a greedy algorithm able to compute superregular matrices $9 \times 9$ and $10 \times 10$ over the field $\F_{2^8}$. The results presented in the tables lead to the following two conjectures.

\begin{conjecture}
For a given $\gamma\geq 2$ and for any odd prime $p$, there exists a vector $(a_1, a_2, \dots ,a_{\gamma-1})\in \F_p^{\gamma-1}$ such that $S_\gamma$ has at most $\frac{N_\gamma}{2}+2$ elements. 
\end{conjecture}

\begin{conjecture}
We also conjecture that for $\gamma\geq 2$, there exists a lower triangular Toeplitz superregular matrix of order $\gamma \times \gamma$ over $\F$ with $|\F|\geq  2^{\frac{2\gamma}{3}}$.
\end{conjecture}

\section{Computer calculations}
In this section, we give a brief description of the computer algorithms we used to obtain the superregular matrices described throughout the paper. All the calculations were performed in Maple.

For $4\leq\gamma\leq 7$, our main goal was to find the minimum number of different minors $A_\gamma\in\F$ has, depending on the finite prime field $\F$. So for each $\gamma$, we started with the smallest possible prime number $p$ and tried all the possible combinations of $(a_3,a_4,\dots,a_{\gamma-1})\in\F_p^{\gamma-3}$ satisfying $a_i\notin S_i$, for $3\leq i\leq\gamma-1$, following the idea explained in Remark \ref{construction}. We were unable to fully achieve the main goal for $\gamma=7$, since when $p\geq 101$ the amount of computations are already too large. For $\gamma=8$ and $\gamma=9$, we just tried to find the two smallest primes $p$ for which exists an LT-superregular Toeplitz matrix over $\F_p$ and gave those examples. For $\gamma=10$, we were unable to find any superregular matrix, using this method of trying all possible values of $(a_3,a_4,\dots,a_{\gamma-1})\in\F_p^{\gamma-3}$, with $p$ small. Therefore for $83\leq p\leq 257$ we randomly select vectors $(a_3,a_4,a_5,a_6)\in\F_p^{4}$, satisfying $a_i\notin S_i$, for $3\leq i\leq 6$  and tried all possible vectors $(a_7,a_8,a_9)\in\F_p^{3}$, satisfying $a_i\notin S_i$, for $7\leq i\leq 9$. In Table \ref{tabA10} we show some examples and the relative frequency of our counts for each of the primes considered.

  \begin{table}[h]
  	\begin{center}
  		\begin{tabular}{|c|c|c|c|}\hline
  			Field size & Example of $(a_3,a_4, \dots,a_\gamma)$ & Different minors & Relative frequency\\ \hline
  			 173 & $(156, 131, 142, 64, 96, 4, 107, 34)$ & 172 & 0.03$\%$\\
  			 193 & $(128, 144, 81, 124, 95, 164, 175, 171)$ & 192 & 0.03$\%$\\
  			 199 & $(179, 172, 149, 3, 168, 93, 129, 187)$ & 198 & 0.03$\%$\\
  			 227 & $(6, 150, 62, 124, 14, 62, 161, 108)$ & 226 & 0.3$\%$\\
  			 229 & $(195, 120, 223, 88, 46, 15, 111, 210)$ & 228 & 1$\%$\\
  			  239 & $(179, 39, 21, 23, 179, 7, 162, 68)$ & 238 & 3$\%$\\
  			   251 & $(131, 135, 195, 56, 39, 64, 185, 43)$ & 250 & 4$\%$\\
  			    257 & $(182, 147, 249, 62, 174, 18, 50, 149)$ & 256 & 5.3$\%$\\\hline
  			
  		\end{tabular}	
  	\end{center}
  	\caption{Examples of LT-superregular matrices of order $10$ for a few small finite prime fields}\label{tabA10}
  \end{table}

\section*{Conclusions and future work}

In this paper we have continued the study of LT-superregular Toeplitz matrices $A_\gamma$. We present new results regarding the minimum number of different minors appearing in $A_\gamma$ and the field sizes that allow the construction of these matrices. Based on the work presented we made the two conjectures. An interesting avenue for further research is to investigate these results using finite field extensions of finite fields of smaller characteristic, e.g., of characteristic $2$, which is of particular interest in coding theory. Another interesting open problem left for future research is to know whether, or in what conditions, there exists LT-superregular matrices over finite fields larger than the minimum $\F_p$ found and smaller than $N_\gamma$. For instance, in \cite{Hutchinson2008} it was found over $\F_{127}$ and in \cite{HaOs2018} over $\F_{256}$,  but to the best of our knowledge the existence of LT-superregular matrices was not known over $\F_p$ when $p\notin \{ 127, 256 \}$ and $p < N_{10} = 2494$. Also nothing is known for the case $\gamma\geq 11$.

\section*{Acknowledgement}

The first author is supported by the Portuguese Foundation for Science and Technology (FCT-Funda\c{c}\~{a}o para a Ci\^{e}ncia e a Tecnologia), through CIDMA - Center for Research and Development in Mathematics and Applications, within project UID/MAT/04106/2019. The second listed author is partially supported by the the Universitat d'Alacant (Grant No. VIGROB-287) and Generalitat Valenciana (Grant No. AICO/2017/128).

\section*{References}

\bibliographystyle{plain}
\bibliography{biblio_com_tudo,code1,Ref-Articles,Ref-Climent-2,Ref-PhDT-MT}

\end{document}